\documentclass[a4paper,11pt,fleqn]{article}
\usepackage{pstricks}
\usepackage{amsmath}
\usepackage{amssymb}
\usepackage{pifont}
\usepackage{theorem} 
\usepackage{euscript}
\usepackage{exscale,relsize}
\usepackage{epic,eepic}
\usepackage{multicol}
\usepackage{makeidx}
\usepackage{srcltx}         
\usepackage{xcolor}
\usepackage{url}
\usepackage{charter}
\usepackage[normalem]{ulem}     
\newlength{\mySubFigSize}
\renewcommand{\leq}{\ensuremath{\leqslant}}
\renewcommand{\geq}{\ensuremath{\geqslant}}

\newcommand{\minimize}[2]{\ensuremath{\underset{\substack{{#1}}}%
{\text{minimize}}\;\;#2 }}

\newcommand{\scal}[2]{{\langle{{#1}\mid{#2}}\rangle}}

\newcommand{\menge}[2]{\big\{{#1}~\big |~{#2}\big\}} 
 
\newcommand{\KKK}{\ensuremath{\boldsymbol{\mathcal K}}}
\newcommand{\HHH}{\ensuremath{\boldsymbol{\mathcal H}}}

\newcommand{\HH}{\ensuremath{{\mathcal H}}}
\newcommand{\GG}{\ensuremath{{\mathcal G}}}

\newcommand{\Sum}{\ensuremath{\displaystyle\sum}}
\newcommand{\emp}{\ensuremath{{\varnothing}}}

\newcommand{\Id}{\ensuremath{\operatorname{Id}}\,}
\newcommand{\cart}{\ensuremath{\raisebox{-0.5mm}{\mbox{\LARGE{$\times$}}}}}

\newcommand{\Tt}{\ensuremath{{\mathfrak{T}}}\,}

\newcommand{\RPP}{\ensuremath{\left]0,+\infty\right[}}

\newcommand{\RX}{\ensuremath{\left]-\infty,+\infty\right]}}

\newcommand{\NN}{\ensuremath{\mathbb N}}

\newcommand{\weakly}{\ensuremath{\:\rightharpoonup\:}}

\newcommand{\ran}{\ensuremath{\text{\rm ran}\,}}

\newcommand{\pinf}{\ensuremath{{+\infty}}}

\newcommand{\prox}{\ensuremath{\text{\rm prox}}}

\newcommand{\gra}{\ensuremath{\text{\rm gra}\,}}

\newcommand{\infconv}{\ensuremath{\mbox{\small$\,\square\,$}}}

\newcommand{\zeroun}{\ensuremath{\left]0,1\right[}}

\newtheorem{theorem}{Theorem}[section]
\newtheorem{lemma}[theorem]{Lemma}
\newtheorem{corollary}[theorem]{Corollary}
\newtheorem{proposition}[theorem]{Proposition}

\theoremstyle{plain}{\theorembodyfont{\rmfamily}%
}
\theoremstyle{plain}{\theorembodyfont{\rmfamily}%
\newtheorem{example}[theorem]{Example}}
\theoremstyle{plain}{\theorembodyfont{\rmfamily}%
\newtheorem{remark}[theorem]{Remark}}
\theoremstyle{plain}{\theorembodyfont{\rmfamily}%
}
\theoremstyle{plain}{\theorembodyfont{\rmfamily}%
}
\theoremstyle{plain}{\theorembodyfont{\rmfamily}%
}
\theoremstyle{plain}{\theorembodyfont{\rmfamily}%
\newtheorem{problem}[theorem]{Problem}}

\numberwithin{equation}{section}
\begin{document}

\title{\sffamily\LARGE Best Approximation from the Kuhn-Tucker Set
of\\ Composite Monotone Inclusions\footnote{Contact author: 
P. L. Combettes, \ttfamily{plc@ljll.math.upmc.fr},
phone: +33 1 4427 6319, fax: +33 1 4427 7200.}}

\author{
Abdullah Alotaibi,$\!^{\natural}$
~Patrick L. Combettes,$\!^{\flat}$~ and 
~Naseer Shahzad$\,^\natural$
\\[5mm]
\small $^\natural$King Abdulaziz University\\
\small Department of Mathematics, P. O. Box 80203\\
\small Jeddah 21859, Saudi Arabia\\[5mm]
\small $^\flat$Sorbonne Universit\'es -- UPMC Univ. Paris 06\\
\small UMR 7598, Laboratoire Jacques-Louis Lions\\
\small F-75005, Paris, France\\[4mm]
}

\date{~}

\maketitle

\vskip 8mm

\begin{abstract} 
\noindent
Kuhn-Tucker points play a fundamental role in the analysis and the
numerical solution of monotone inclusion problems, providing in
particular both primal and dual solutions. We propose a class of 
strongly convergent algorithms for constructing the best
approximation to a reference point from the set of Kuhn-Tucker 
points of a general Hilbertian composite monotone inclusion
problem. Applications to systems of coupled monotone inclusions are
presented. Our framework does not impose additional assumptions on 
the operators present in the formulation, and it does not require
knowledge of the norm of the linear operators involved in the
compositions or the inversion of linear operators.
\end{abstract} 



\maketitle

\newpage

\section{Introduction}

Let $\HH$ and $\GG$ be real Hilbert spaces, let $L\colon\HH\to\GG$
be a bounded linear operator, and let $f\colon\HH\to\RX$ and 
$g\colon\GG\to\RX$ be proper lower semicontinuous convex functions.
Classical Fenchel-Rockafellar duality \cite{Rock67} concerns the
interplay between the optimization problem 
\begin{equation}
\label{e:9F842h20a}
\minimize{x\in\HH}{f(x)+g(Lx)}
\end{equation}
and its dual 
\begin{equation}
\label{e:9F842h20b}
\minimize{v^*\in\GG}{f^*(-L^*v^*)+g^*(v^*)}.
\end{equation}
An essential ingredient in the analysis of such dual problems 
is the associated Kuhn-Tucker set \cite{Rock74}
\begin{equation}
\label{e:9F842h20c}
\boldsymbol{Z}=\menge{(x,v^*)\in\HH\oplus\GG}
{-L^*v^*\in\partial f(x)\:\;\text{and}\;Lx\in\partial g^*(v^*)},
\end{equation}
which involves the maximally monotone subdifferential operators 
$\partial f$ and $\partial g^*$.
A fruitful generalization of 
\eqref{e:9F842h20a}--\eqref{e:9F842h20b} is obtained by
pairing the inclusion $0\in Ax+L^*BLx$ on $\HH$ with the dual
inclusion $0\in -LA^{-1}(-L^*v^*)+B^{-1}v^*$ on $\GG$, where $A$
and $B$ are maximally monotone operators acting on $\HH$ and $\GG$,
respectively. Such operator duality has been studied in 
\cite{Ecks99,Penn00,Robi99,Robi01} and the first splitting
algorithm for solving such composite inclusions was proposed in
\cite{Siop11}. The strategy adopted in that paper was to use a
standard 2-operator splitting method to construct a point in the 
Kuhn-Tucker set $\boldsymbol{Z}=\menge{(x,v^*)\in\HH\oplus\GG}
{-L^*v^*\in Ax\:\;\text{and}\;Lx\in B^{-1}v^*}$ and hence obtain a 
primal-dual solution (see also
\cite{Bot13a,Siop13,Svva12,Opti13,Bang13} for variants of this
approach). In \cite{Genn13} we investigated a different strategy
based on an idea first proposed in \cite{Svai08} for solving the
inclusion $0\in Ax+Bx$. In this framework, at each iteration,
one uses points in the graphs of $A$ and $B$ to construct a closed 
affine half-space of $\HH\oplus\GG$ containing $\boldsymbol{Z}$; 
the primal-dual update is then obtained as the
projection of the current iterate onto it.
The resulting Fej\'er-monotone algorithm provides only weak
convergence to an unspecified Kuhn-Tucker point. In the present
paper we propose a strongly convergent modification of these
methods for solving the following best approximation problem.   

\begin{problem}
\label{prob:11}
Let $\HH$ and $\GG$ be real Hilbert spaces, and set 
$\KKK=\HH\oplus\GG$. Let 
$A\colon\HH\to 2^{\HH}$ and $B\colon\GG\to 2^{\GG}$ be maximally 
monotone operators, and let $L\colon\HH\to\GG$ be a bounded linear 
operator. Let $(x_0,v_0^*)\in\KKK$, assume that the inclusion
problem
\begin{equation}
\label{e:primal}
\text{find}\;\;x\in\HH\;\;\text{such that}\;\;0\in Ax+L^*BLx
\end{equation}
has at least one solution, and consider the dual problem 
\begin{equation}
\label{e:dual}
\text{find}\;\;v^*\in\GG\;\;\text{such that}\;\; 
0\in -LA^{-1}(-L^*v^*)+B^{-1}v^*.
\end{equation}
The problem is to find the best approximation 
$(\overline{x},\overline{v}^*)$ to $(x_0,v_0^*)$ 
from the associated Kuhn-Tucker set
\begin{equation}
\label{e:Aas7B09k21a}
\boldsymbol{Z}=\menge{(x,v^*)\in\KKK}
{-L^*v^*\in Ax\:\;\text{and}\;Lx\in B^{-1}v^*}.
\end{equation}
\end{problem}

The principle of our algorithm goes back to the work of Yves 
Haugazeau \cite{Haug68} for finding the projection of a point 
onto the intersection of closed convex sets by means of projections
onto the individual sets. Haugazeau's method was generalized in
several directions and applied to a variety of problems in nonlinear
analysis and optimization in \cite{Sico00}. In \cite{Moor01}, it 
was formulated as an abstract convergence principle for turning a 
class of weakly convergent methods into strongly convergent ones 
(see also \cite{Marq13} for recent related work). In the area of 
monotone inclusions, Haugazeau-like
methods were used in \cite{Solo00} for solving $x\in A^{-1}0$ and
in \cite{Moor01} for solving $x\in\bigcap_{i=1}^mA_i^{-1}0$. They
were also used in splitting method for solving $0\in Ax+Bx$ as a
modification of the forward-backward splitting algorithm in
\cite{Hirs05} and \cite[Corollary~29.5]{Livre1}, and as a
modification of the Douglas-Rachford algorithm in 
\cite{Joat06} and \cite{Zhan13}. 

The paper is organized as follows. Section~\ref{sec:2} is devoted
to a version of an abstract Haugazeau principle. The algorithms for 
solving Problem~\ref{prob:11} are presented in Section~\ref{sec:3},
where their strong convergence is established.
In Section~\ref{sec:4}, we present an extension to systems of
coupled monotone inclusions and consider applications to the
relaxation of inconsistent common zero problems and to structured
multivariate convex minimization problems. 

\noindent
{\bfseries Notation.}
Our notation is standard and follows \cite{Livre1}, where the
necessary background on monotone operators and convex analysis is
available. The scalar product of a Hilbert space is denoted by 
$\scal{\cdot}{\cdot}$ and the associated norm by $\|\cdot\|$.
We denote respectively by $\weakly$ and $\to$ weak and strong 
convergence, and by $\Id$ the identity operator. Let $\HH$ and 
$\GG$ be real Hilbert space. The Hilbert direct sum of $\HH$ and 
$\GG$ is denoted by $\HH\oplus\GG$, and the power set of $\HH$ by
$2^{\HH}$. Now let $A\colon\HH\to 2^{\HH}$. Then
$\ran A$ is the range $A$, $\gra A$ the graph of $A$, 
$A^{-1}$ the inverse of $A$, and $J_A=(\Id+A)^{-1}$ the
resolvent of $A$. The projection operator onto a nonempty closed 
convex subset $C$ of $\HH$ is denoted by $P_C$ and  
$\Gamma_0(\HH)$ is the class of proper lower semicontinuous convex 
functions from $\HH$ to $\RX$. 
Let $f\in\Gamma_0(\HH)$. The conjugate of $f$ is 
$\Gamma_0(\HH)\ni f^*\colon u^*\mapsto
\sup_{x\in\HH}(\scal{x}{u^*}-f(x))$ and the subdifferential of $f$
is $\partial f\colon\HH\to 2^{\HH}\colon x\mapsto\menge{u^*\in\HH}
{(\forall y\in\HH)\;\:\scal{y-x}{u^*}+f(x)\leq f(y)}$.

\section{An abstract Haugazeau algorithm}
\label{sec:2}

In \cite[Th\'eor\`eme~3-2]{Haug68} Haugazeau proposed an ingenious 
method for projecting a point onto the intersection of closed 
convex sets in a Hilbert space using the projections onto the 
individual sets. Abstract versions of his method for projecting
onto a closed convex set in a real Hilbert space 
were devised in \cite{Sico00} and \cite{Moor01}. In this 
section, we present a formulation of this abstract principle which 
is better suited for our purposes.

Let $\HHH$ be a real Hilbert space. Given an ordered triplet 
$(\boldsymbol{x},\boldsymbol{y},\boldsymbol{z})\in\HHH^3$, we 
define
\begin{equation}
H(\boldsymbol{x},\boldsymbol{y})=\menge{\boldsymbol{h}\in\HHH}
{\scal{\boldsymbol{h}-\boldsymbol{y}}
{\boldsymbol{x}-\boldsymbol{y}}\leq 0}.
\end{equation}
Moreover, if $\boldsymbol{R}=H(\boldsymbol{x},\boldsymbol{y})\cap
H(\boldsymbol{y},\boldsymbol{z})\neq\emp$, we denote by
$Q(\boldsymbol{x},\boldsymbol{y},\boldsymbol{z})$ the 
projection of $\boldsymbol{x}$ onto $\boldsymbol{R}$. The principle 
of the algorithm to project a point $\boldsymbol{x}_0\in\HHH$ onto 
a nonempty closed convex set $\boldsymbol{C}\subset\HHH$ is to 
use at iteration $n$ the current iterate $\boldsymbol{x}_n$ to 
construct an outer approximation to $\boldsymbol{C}$ of the form 
$H(\boldsymbol{x}_0,\boldsymbol{x}_n)\cap 
H(\boldsymbol{x}_n,\boldsymbol{x}_{n+1/2})$; the update is then 
computed as the projection of $\boldsymbol{x}_0$ onto it, i.e., 
$\boldsymbol{x}_{n+1}=Q(\boldsymbol{x}_0,\boldsymbol{x}_n,
\boldsymbol{x}_{n+1/2})$. 

\begin{proposition}
\label{p:j7yG9i-9jmL406}
Let $\boldsymbol{C}$ be a nonempty closed convex subset of $\HHH$ 
and let $\boldsymbol{x}_0\in\HHH$. Iterate
\begin{equation}
\label{e:Aas7B09k20f}
\begin{array}{l}
\text{for}\;n=0,1,\ldots\\
\left\lfloor
\begin{array}{l}
\text{take}\;\boldsymbol{x}_{n+1/2}\in\HHH\;\text{such that}\;
\boldsymbol{C}\subset H(\boldsymbol{x}_n,\boldsymbol{x}_{n+1/2})\\
\boldsymbol{x}_{n+1}=Q\big(\boldsymbol{x}_0,\boldsymbol{x}_n,
\boldsymbol{x}_{n+1/2}\big).
\end{array}
\right.\\
\end{array}
\end{equation}
Then the sequence $(\boldsymbol{x}_n)_{n\in\NN}$ is well defined 
and the following hold:
\begin{enumerate}
\item
\label{p:j7yG9i-9jmL406i}
$(\forall n\in\NN)$ $\boldsymbol{C}\subset
H(\boldsymbol{x}_0,\boldsymbol{x}_n)\cap
H(\boldsymbol{x}_n,\boldsymbol{x}_{n+1/2})$.
\item
\label{p:j7yG9i-9jmL406ii}
$\sum_{n\in\NN}\|\boldsymbol{x}_{n+1}-
\boldsymbol{x}_n\|^2<\pinf$.
\item
\label{p:j7yG9i-9jmL406iii}
$\sum_{n\in\NN}\|\boldsymbol{x}_{n+1/2}-
\boldsymbol{x}_n\|^2<\pinf$.
\item
\label{p:j7yG9i-9jmL406iv}
Suppose that, for every $\boldsymbol{x}\in\HHH$ and every 
strictly increasing sequence $(k_n)_{n\in\NN}$ in $\NN$, 
$\boldsymbol{x}_{k_n}\weakly\boldsymbol{x}$
$\Rightarrow$ $\boldsymbol{x}\in\boldsymbol{C}$.
Then $\boldsymbol{x}_n\to P_{\boldsymbol{C}}\boldsymbol{x}_0$.
\end{enumerate}
\end{proposition}
\begin{proof}
The proof is similar to those found in \cite[Section~3]{Moor01} 
and \cite[Section~3]{Sico00}. First, recall that the projector onto 
a nonempty closed convex subset $\boldsymbol{D}$ of $\HHH$ is 
characterized by \cite[Theorem~3.14]{Livre1}
\begin{equation}
\label{e:our-old-mor2001}
(\forall\boldsymbol{x}\in\HH)\quad P_{\boldsymbol{D}}\boldsymbol{x}
\in\boldsymbol{D}\quad\text{and}\quad\boldsymbol{D}\subset 
H(\boldsymbol{x},P_{\boldsymbol{D}}\boldsymbol{x}).
\end{equation}

\ref{p:j7yG9i-9jmL406i}:
Let $n\in\NN$ be such that $\boldsymbol{x}_n$ exists. Since by 
construction $\boldsymbol{C}\subset
H(\boldsymbol{x}_n,\boldsymbol{x}_{n+1/2})$,
it is enough to show that $\boldsymbol{C}\subset
H(\boldsymbol{x}_0,\boldsymbol{x}_n)$. This inclusion is
trivially true for $n=0$ since
$H(\boldsymbol{x}_0,\boldsymbol{x}_0)=\HHH$. 
Furthermore, it follows from \eqref{e:our-old-mor2001} and
\eqref{e:Aas7B09k20f} that
\begin{eqnarray}
\label{e:xenon}
\boldsymbol{C}\subset H(\boldsymbol{x}_0,\boldsymbol{x}_n)
&\Rightarrow&\boldsymbol{C}\subset 
H(\boldsymbol{x}_0,\boldsymbol{x}_n)\cap 
H(\boldsymbol{x}_n,\boldsymbol{x}_{n+1/2})\nonumber\\
&\Rightarrow&\boldsymbol{C}\subset 
H\big(\boldsymbol{x}_0,Q(\boldsymbol{x}_0,
\boldsymbol{x}_n,\boldsymbol{x}_{n+1/2})\big)\nonumber\\
&\Leftrightarrow&\boldsymbol{C}\subset 
H(\boldsymbol{x}_0,\boldsymbol{x}_{n+1}),
\end{eqnarray}
which establishes the assertion by induction.
This also shows that $H(\boldsymbol{x}_0,\boldsymbol{x}_n)\cap
H(\boldsymbol{x}_n,\boldsymbol{x}_{n+1/2})$ is a nonempty closed
convex set and therefore that the projection 
$\boldsymbol{x}_{n+1}$ of $\boldsymbol{x}_0$ onto it is well 
defined.

\ref{p:j7yG9i-9jmL406ii}:
Let $n\in\NN$. By construction, 
$\boldsymbol{x}_{n+1}=Q(\boldsymbol{x}_0,\boldsymbol{x}_n,
\boldsymbol{x}_{n+1/2})\in H(\boldsymbol{x}_0,\boldsymbol{x}_n)\cap
H\big(\boldsymbol{x}_n,\boldsymbol{x}_{n+1/2}\big)$. Consequently, 
since $\boldsymbol{x}_n$ is the projection of $\boldsymbol{x}_0$ 
onto $H(\boldsymbol{x}_0,\boldsymbol{x}_n)$ and 
$\boldsymbol{x}_{n+1}\in H(\boldsymbol{x}_0,\boldsymbol{x}_n)$, 
we have $\|\boldsymbol{x}_0-\boldsymbol{x}_n\|\leq
\|\boldsymbol{x}_0-\boldsymbol{x}_{n+1}\|$.
On the other hand, since $P_{\boldsymbol{C}}\boldsymbol{x}_0\in
\boldsymbol{C}\subset H(\boldsymbol{x}_0,\boldsymbol{x}_n)$, 
we have $\|\boldsymbol{x}_0-\boldsymbol{x}_n\|\leq
\|\boldsymbol{x}_0-P_{\boldsymbol{C}}\boldsymbol{x}_0\|$.
It follows that $(\|\boldsymbol{x}_0-\boldsymbol{x}_k\|)_{k\in\NN}$ 
converges and that
\begin{equation}
\label{e:KKn+=2-07g}
\lim\|\boldsymbol{x}_0-\boldsymbol{x}_k\|\leq
\|\boldsymbol{x}_0-P_{\boldsymbol{C}}\boldsymbol{x}_0\|.
\end{equation}
On the other hand, since
$\boldsymbol{x}_{n+1}\in H(\boldsymbol{x}_0,\boldsymbol{x}_n)$,
we have
\begin{equation}
\|\boldsymbol{x}_{n+1}-\boldsymbol{x}_n\|^2
\leq\|\boldsymbol{x}_{n+1}-\boldsymbol{x}_n\|^2
+2\scal{\boldsymbol{x}_{n+1}-\boldsymbol{x}_n}
{\boldsymbol{x}_n-\boldsymbol{x}_0}
=\|\boldsymbol{x}_0-\boldsymbol{x}_{n+1}\|^2-
\|\boldsymbol{x}_0-\boldsymbol{x}_n\|^2.
\end{equation}
Hence, $\sum_{k=1}^n\|\boldsymbol{x}_{k+1}-\boldsymbol{x}_k\|^2
\leq\|\boldsymbol{x}_0-\boldsymbol{x}_{n+1}\|^2\leq
\|\boldsymbol{x}_0-P_{\boldsymbol{C}}\boldsymbol{x}_0\|^2$
and, in turn, 
$\sum_{k\in\NN}\|\boldsymbol{x}_{k+1}-\boldsymbol{x}_k\|^2<\pinf$.

\ref{p:j7yG9i-9jmL406iii}:
For every $n\in\NN$, we derive from the inclusion 
$\boldsymbol{x}_{n+1}\in 
H(\boldsymbol{x}_n,\boldsymbol{x}_{n+1/2})$ that
\begin{align}
\label{e:KKn+=2-07m}
\|\boldsymbol{x}_{n+1/2}-\boldsymbol{x}_n\|^2
&\leq\|\boldsymbol{x}_{n+1}-\boldsymbol{x}_{n+1/2}\|^2+
\|\boldsymbol{x}_n-\boldsymbol{x}_{n+1/2}\|^2\nonumber\\
&\leq\|\boldsymbol{x}_{n+1}-\boldsymbol{x}_{n+1/2}\|^2+
2\scal{\boldsymbol{x}_{n+1}-\boldsymbol{x}_{n+1/2}}
{\boldsymbol{x}_{n+1/2}-\boldsymbol{x}_n}
+\|\boldsymbol{x}_n-\boldsymbol{x}_{n+1/2}\|^2\nonumber\\
&=\|\boldsymbol{x}_{n+1}-\boldsymbol{x}_n\|^2.
\end{align}
Hence, it follows from \ref{p:j7yG9i-9jmL406ii} that
$\sum_{n\in\NN}\|\boldsymbol{x}_{n+1/2}-\boldsymbol{x}_n\|^2<\pinf$. 

\ref{p:j7yG9i-9jmL406iv}:
Let us note that \eqref{e:KKn+=2-07g} implies that
$(\boldsymbol{x}_n)_{n\in\NN}$ is bounded. Now, let
$\boldsymbol{x}$ be a weak sequential cluster point of 
$(\boldsymbol{x}_n)_{n\in\NN}$, say $\boldsymbol{x}_{k_n}\weakly
\boldsymbol{x}$. Then, by weak lower semicontinuity of 
$\|\cdot\|$ \cite[Lemma~2.35]{Livre1} and \eqref{e:KKn+=2-07g}
$\|\boldsymbol{x}_0-\boldsymbol{x}\|\leq\varliminf
\|\boldsymbol{x}_0-\boldsymbol{x}_{k_n}\|\leq
\|\boldsymbol{x}_0-P_{\boldsymbol{C}}\boldsymbol{x}_0\|=
\inf_{\boldsymbol{y}\in\boldsymbol{C}}
\|\boldsymbol{x}_0-\boldsymbol{y}\|$.
Hence, since $\boldsymbol{x}\in\boldsymbol{C}$,
$\boldsymbol{x}=P_{\boldsymbol{C}}\boldsymbol{x}_0$ is the only 
weak sequential cluster point of the sequence 
$(\boldsymbol{x}_n)_{n\in\NN}$ and it follows from 
\cite[Lemma~2.38]{Livre1} that 
$\boldsymbol{x}_n\weakly P_{\boldsymbol{C}}\boldsymbol{x}_0$.
In turn \eqref{e:KKn+=2-07g} yields
$\|\boldsymbol{x}_0-P_{\boldsymbol{C}}\boldsymbol{x}_0\|\leq
\varliminf\|\boldsymbol{x}_0-\boldsymbol{x}_n\|=
\lim\|\boldsymbol{x}_0-\boldsymbol{x}_n\|
\leq\|\boldsymbol{x}_0-P_{\boldsymbol{C}}\boldsymbol{x}_0\|$.
Thus, $\boldsymbol{x}_0-\boldsymbol{x}_n\weakly
\boldsymbol{x}_0-P_{\boldsymbol{C}}\boldsymbol{x}_0$ and
$\|\boldsymbol{x}_0-\boldsymbol{x}_n\|\to
\|\boldsymbol{x}_0-P_{\boldsymbol{C}}\boldsymbol{x}_0\|$.
We therefore derive from \cite[Lemma~2.41(i)]{Livre1} that 
$\boldsymbol{x}_0-\boldsymbol{x}_n\to
\boldsymbol{x}_0-P_{\boldsymbol{C}}\boldsymbol{x}_0$, i.e., 
$\boldsymbol{x}_n\to P_{\boldsymbol{C}}\boldsymbol{x}_0$.
\end{proof}

\begin{remark}
\label{r:9i-9jmL428}
Suppose that, for some $n\in\NN$, $\boldsymbol{x}_n\in
\boldsymbol{C}$ in \eqref{e:Aas7B09k20f}. Then 
$\|\boldsymbol{x}_0-P_{\boldsymbol{C}}\boldsymbol{x}_0\|\leq
\|\boldsymbol{x}_0-\boldsymbol{x}_n\|$ and, since we always have
$\|\boldsymbol{x}_0-\boldsymbol{x}_n\|\leq\|\boldsymbol{x}_0-
P_{\boldsymbol{C}}\boldsymbol{x}_0\|$, we conclude that 
$\boldsymbol{x}_n=P_{\boldsymbol{C}}\boldsymbol{x}_0$
and that the iterations can be stopped.
\end{remark}

Algorithm \eqref{e:Aas7B09k20f} can easily be implemented thanks 
to the following lemma.

\begin{lemma}
\label{l:haugazeauy}
Let $(\boldsymbol{x},\boldsymbol{y},\boldsymbol{z})\in\HHH^3$ 
and set $\boldsymbol{R}=H(\boldsymbol{x},\boldsymbol{y})\cap
H(\boldsymbol{y},\boldsymbol{z})$.
Moreover, set
$\chi=\scal{\boldsymbol{x}-\boldsymbol{y}}{\boldsymbol{y}-
\boldsymbol{z}}$, $\mu=\|\boldsymbol{x}-\boldsymbol{y}\|^2$,
$\nu=\|\boldsymbol{y}-\boldsymbol{z}\|^2$, and
$\rho=\mu\nu-\chi^2$.
Then exactly one of the following holds:
\begin{enumerate}
\item
\label{c:haugazeaui}
$\rho=0$ and $\chi<0$, in which case $\boldsymbol{R}=\emp$.
\item 
\label{c:haugazeauii}
\emph{[}$\,\rho=0$ and $\chi\geq 0\,$\emph{]} or 
$\rho>0$, in which case $\boldsymbol{R}\neq\emp$ and 
\begin{equation}
\label{e:j7yG9i-9jmL405}
Q(\boldsymbol{x},\boldsymbol{y},\boldsymbol{z})=
\begin{cases}
\boldsymbol{z}, &\!\text{if}\;\rho=0\;\text{and}\;
\chi\geq 0;\\[+0mm]
\displaystyle
\boldsymbol{x}+(1+\chi/\nu)
(\boldsymbol{z}-\boldsymbol{y}), 
&\!\text{if}\;\rho>0\;\text{and}\;
\chi\nu\geq\rho;\\
\displaystyle \boldsymbol{y}+(\nu/\rho)
\big(\chi(\boldsymbol{x}-\boldsymbol{y})
+\mu(\boldsymbol{z}-\boldsymbol{y})\big), 
&\!\text{if}\;\rho>0\;\text{and}\;\chi\nu<\rho.
\end{cases}
\end{equation}
\end{enumerate}
\end{lemma}
\begin{proof}
See \cite[Th\'eor\`eme~3-1]{Haug68} for the original proof and 
\cite[Corollary~28.21]{Livre1} for an alternate derivation.
\end{proof}

\section{Main result}
\label{sec:3}

In this section, we devise a strongly convergent algorithm for
solving Problem~\ref{prob:11} by coupling
Proposition~\ref{p:j7yG9i-9jmL406} with the 
construction of \cite{Genn13} to determine the half-spaces 
$(H(\boldsymbol{x}_n,\boldsymbol{x}_{n+1/2}))_{n\in\NN}$. 
First, we need a couple of facts. 

\begin{proposition}{\rm \cite[Proposition~2.8]{Siop11}}
\label{p:0kkhUj710-31z}
In the setting of Problem~\ref{prob:11}, 
$\boldsymbol{Z}$ is a nonempty closed convex set and, if 
$(x,v^*)\in\boldsymbol{Z}$, then $x$ solves \eqref{e:primal} and 
$v^*$ solves \eqref{e:dual}.
\end{proposition}

\begin{proposition}{\rm \cite[Proposition~2.5]{Genn13}}
\label{p:genna3Hbl-915}
In the setting of Problem~\ref{prob:11}, let $(a_n,a_n^*)_{n\in\NN}$
be a sequence in $\gra A$, let $(b_n,b_n^*)_{n\in\NN}$ be a 
sequence in $\gra B$, and let $(x,v^*)\in\KKK$. Suppose that 
$a_n\weakly{x}$, $b^*_n\weakly{v}^*$, $a^*_n+L^*b^*_n\to 0$, 
and $La_n-b_n\to 0$. Then $\scal{a_n}{a_n^*}+\scal{b_n}{b_n^*}\to 0$
and $(x,v^*)\in\boldsymbol{Z}$. 
\end{proposition}

The next result features our general algorithm for solving 
Problem~\ref{prob:11}.

\begin{theorem}
\label{t:9i-9jmL402}
Consider the setting of Problem~\ref{prob:11}. Let 
$\varepsilon\in\zeroun$, let $\alpha\in\RPP$, and set, for every
$(x,v^*)\in\KKK$, 
\begin{multline}
\label{e:9g45g2h29a}
\boldsymbol{G}_\alpha(x,v^*)=
\Big\{(a,b,a^*,b^*)\in\KKK\times\KKK\;\big |\;
(a,a^*)\in\gra A,\;(b,b^*)\in\gra B,\;\text{and}\\
\scal{x-a}{a^*+L^*v^*}+\scal{Lx-b}{b^*-v^*}
\geq\alpha\big(\|a^*+L^*b^*\|^2+\|La-b\|^2\big)\Big\}.
\end{multline}
Iterate
\begin{equation}
\label{e:j7yG9i-9jmL409h}
\begin{array}{l}
\text{for}\;n=0,1,\ldots\\
\left\lfloor
\begin{array}{l}
(a_n,b_n,a_n^*,b_n^*)\in\boldsymbol{G}_\alpha(x_n,v_n^*)\\
s^*_n=a^*_n+L^*b^*_n\\
t_n=b_n-La_n\\
\tau_n=\|s_n^*\|^2+\|t_n\|^2\\
\text{if}\;\tau_n=0\\
\left\lfloor
\begin{array}{l}
\theta_n=0\\
\end{array}
\right.\\
\text{if}\;\tau_n>0\\
\left\lfloor
\begin{array}{l}
\lambda_n\in\left[\varepsilon,1\right]\\
\theta_n=\lambda_n\big(\scal{x_n}{s^*_n}+\scal{t_n}{v^*_n}
-\scal{a_n}{a^*_n}-\scal{b_n}{b^*_n}\big)/\tau_n\\
\end{array}
\right.\\
x_{n+1/2}=x_n-\theta_n s^*_n\\
v^*_{n+1/2}=v^*_n-\theta_n t_n\\
\chi_n=\scal{x_0-x_n}{x_n-x_{n+1/2}}
+\scal{v_0^*-v_n^*}{v_n^*-v_{n+1/2}^*}\\
\mu_n=\|x_0-x_n\|^2+\|v_0^*-v_n^*\|^2\\
\nu_n=\|x_n-x_{n+1/2}\|^2+\|v_n^*-v_{n+1/2}^*\|^2\\
\rho_n=\mu_n\nu_n-\chi_n^2\\
\text{if}\;\rho_n=0\;\text{and}\;\chi_n\geq 0\\
\left\lfloor
\begin{array}{l}
x_{n+1}=x_{n+1/2}\\
v^*_{n+1}=v_{n+1/2}^*
\end{array}
\right.\\
\text{if}\;\rho_n>0\;\text{and}\;\chi_n\nu_n\geq\rho_n\\
\left\lfloor
\begin{array}{l}
x_{n+1}=x_0+(1+\chi_n/\nu_n)(x_{n+1/2}-x_n)\\
v^*_{n+1}=v_0^*+(1+\chi_n/\nu_n)(v_{n+1/2}^*-v_n^*)
\end{array}
\right.\\
\text{if}\;\rho_n>0\;\text{and}\;\chi_n\nu_n<\rho_n\\
\left\lfloor
\begin{array}{l}
x_{n+1}=x_n+(\nu_n/\rho_n)\big(\chi_n(x_0-x_n)
+\mu_n(x_{n+1/2}-x_n)\big)\\
v^*_{n+1}=v_n^*+(\nu_n/\rho_n)\big(\chi_n(v_0^*-v_n^*)
+\mu_n(v_{n+1/2}^*-v_n^*)\big).
\end{array}
\right.\\
\end{array}
\right.\\
\end{array}
\end{equation}
Then \eqref{e:j7yG9i-9jmL409h} generates infinite sequences 
$(x_n)_{n\in\NN}$ and $(v_n^*)_{n\in\NN}$, and the following
hold:
\begin{enumerate}
\item
\label{t:9i-9jmL402i}
$\sum_{n\in\NN}\|x_{n+1}-x_n\|^2<\pinf$ and 
$\sum_{n\in\NN}\|v^*_{n+1}-v^*_n\|^2<\pinf$.
\item
\label{t:9i-9jmL402ii}
$\sum_{n\in\NN}\|s^*_n\|^2<\pinf$ and 
$\sum_{n\in\NN}\|t_n\|^2<\pinf$.
\item
\label{t:9i-9jmL402iii}
Suppose that $x_n-a_n\weakly 0$ and $v_n^*-b_n^*\weakly 0$. Then 
$x_n\to\overline{x}$ and $v_n^*\to\overline{v}^*$.
\end{enumerate}
\end{theorem}
\begin{proof}
We are going to show that the claims follow from
Proposition~\ref{p:j7yG9i-9jmL406} applied in $\KKK$ to the
set $\boldsymbol{Z}$ of \eqref{e:Aas7B09k21a}, which is nonempty, 
closed, and convex by Proposition~\ref{p:0kkhUj710-31z}. First, 
let us set
\begin{equation}
\label{e:j7yG9i-9jmL411b}
(\forall n\in\NN)\quad
\boldsymbol{x}_n=(x_n,v_n^*)\quad\text{and}\quad
\boldsymbol{x}_{n+1/2}=(x_{n+1/2},v_{n+1/2}^*).
\end{equation}
We deduce from \eqref{e:j7yG9i-9jmL409h} that 
\begin{align}
\label{e:j7yG9i-9jmL410x}
&\hskip -4mm
(\forall (x,v^*)\in\KKK)(\forall n\in\NN)\quad
\scal{x}{s^*_n}+\scal{t_n}{v^*}-\scal{a_n}{a^*_n}-
\scal{b_n}{b^*_n}\nonumber\\
&\hskip 44mm =\scal{x}{a^*_n+L^*b^*_n}+\scal{b_n-La_n}{v^*}-
\scal{a_n}{a^*_n}-\scal{b_n}{b^*_n}\nonumber\\
&\hskip 44mm =\scal{x-a_n}{a_n^*+L^*v^*}+
\scal{Lx-b_n}{b_n^*-v^*}.
\end{align}
Next, let us show that 
\begin{equation}
\label{e:j7yG9i-9jmL411u}
(\forall n\in\NN)\quad\boldsymbol{Z}\subset
H\big(\boldsymbol{x}_n,\boldsymbol{x}_{n+1/2}\big).
\end{equation}
To this end, let $\boldsymbol{z}=(x,v^*)\in\boldsymbol{Z}$ and let 
$n\in\NN$. We must show that
$\scal{\boldsymbol{z}-\boldsymbol{x}_{n+1/2}}{\boldsymbol{x}_{n}-
\boldsymbol{x}_{n+1/2}}\leq 0$. If $\tau_n=0$, then
$\boldsymbol{x}_{n+1/2}=\boldsymbol{x}_{n}$ and the inequality
is trivially satisfied. Now suppose that $\tau_n>0$.
Then \eqref{e:j7yG9i-9jmL410x} and
\eqref{e:9g45g2h29a} yield
\begin{align}
\label{e:9i-9jmL417}
\theta_n
&=\lambda_n\frac{\scal{x_n}{s^*_n}+\scal{t_n}{v^*_n}
-\scal{a_n}{a^*_n}-\scal{b_n}{b^*_n}}{\tau_n}\nonumber\\
&=\lambda_n\frac{\scal{x_n-a_n}{a_n^*+L^*v_n^*}+
\scal{Lx_n-b_n}{b_n^*-v_n^*}}{\tau_n}\nonumber\\
&\geq\varepsilon\alpha\nonumber\\
&>0.
\end{align}
On the other hand, it follows from \eqref{e:j7yG9i-9jmL409h} 
and \eqref{e:Aas7B09k21a} that $a_n^*\in Aa_n$ and 
$-L^*v^*\in Ax$. Hence, since $A$ is monotone, 
$\scal{x-a_n}{a_n^*+L^*v^*}\leq 0$. 
Similarly, since $v^*\in B(Lx)$ and $b_n^*\in Bb_n$, the 
monotonicity of $B$ implies that $\scal{Lx-b_n}{b_n^*-v^*}\leq 0$.
Consequently, we derive from \eqref{e:j7yG9i-9jmL409h},
\eqref{e:j7yG9i-9jmL410x}, and \eqref{e:9g45g2h29a} that 
\begin{align}
\label{e:9g45g2h30}
&\hskip -6mm
\scal{\boldsymbol{z}-\boldsymbol{x}_{n+1/2}}{\boldsymbol{x}_{n}-
\boldsymbol{x}_{n+1/2}}/\theta_n\nonumber\\
&=\scal{\boldsymbol{z}}{\boldsymbol{x}_n-\boldsymbol{x}_{n+1/2}}/
\theta_n+\scal{\boldsymbol{x}_{n+1/2}}
{\boldsymbol{x}_{n+1/2}-\boldsymbol{x}_n}/\theta_n\nonumber\\
&=\scal{x}{x_n-x_{n+1/2}}/\theta_n+
\scal{v^*}{v_n^*-v_{n+1/2}^*}/\theta_n\nonumber\\
&\quad\;+\scal{x_{n+1/2}}{x_{n+1/2}-x_n}/\theta_n+
\scal{v^*_{n+1/2}}{v^*_{n+1/2}-v^*_n}/\theta_n\nonumber\\
&=\scal{x}{s_n^*}+\scal{t_n}{v^*}-\scal{x_n}{s_n^*}-
\scal{t_n}{v^*_n}+\theta_n\big(\|s_n^*\|^2+\|t_n\|^2\big)
\nonumber\\
&=\scal{x}{s_n^*}+\scal{t_n}{v^*}
-\scal{x_n}{s_n^*}-\scal{t_n}{v^*_n}+
\lambda_n\big(\scal{x_n}{s^*_n}+\scal{t_n}{v^*_n}
-\scal{a_n}{a^*_n}-\scal{b_n}{b^*_n}\big)
\nonumber\\
&=\scal{x}{s_n^*}+\scal{t_n}{v^*}
-\scal{a_n}{a_n^*}-\scal{b_n}{b_n^*}\nonumber\\
&\quad\;-(1-\lambda_n)\big(\scal{x_n}{s^*_n}+\scal{t_n}{v^*_n}
-\scal{a_n}{a_n^*}-\scal{b_n}{b_n^*}\big)\nonumber\\
&=\scal{x-a_n}{a_n^*+L^*v^*}+\scal{Lx-b_n}{b_n^*-v^*}\nonumber\\
&\quad\;-(1-\lambda_n)\big(\scal{x_n-a_n}{a_n^*+L^*v_n^*}+
\scal{Lx_n-b_n}{b_n^*-v_n^*}\big)\nonumber\\
&\leq\scal{x-a_n}{a_n^*+L^*v^*}+\scal{Lx-b_n}{b_n^*-v^*}
-\alpha(1-\lambda_n)\big(\|a_n^*+L^*b_n^*\|^2+
\|La_n-b_n\|^2\big)\nonumber\\
&\leq\scal{x-a_n}{a_n^*+L^*v^*}+\scal{Lx-b_n}{b_n^*-v^*}\nonumber\\
&\leq 0.
\end{align}
This verifies \eqref{e:j7yG9i-9jmL411u}. It therefore follows 
from \eqref{e:j7yG9i-9jmL405} that 
\eqref{e:j7yG9i-9jmL409h} is an instance of 
\eqref{e:Aas7B09k20f}.

\ref{t:9i-9jmL402i}:
It follows from \eqref{e:j7yG9i-9jmL411b} and 
Proposition~\ref{p:j7yG9i-9jmL406}\ref{p:j7yG9i-9jmL406ii}
that $\sum_{n\in\NN}\|x_{n+1}-x_n\|^2+
\sum_{n\in\NN}\|v^*_{n+1}-v^*_n\|^2=
\sum_{n\in\NN}\|\boldsymbol{x}_{n+1}-
\boldsymbol{x}_n\|^2<\pinf$.

\ref{t:9i-9jmL402ii}:
Let $n\in\NN$. We consider two cases.
\begin{itemize}
\item
$\tau_n=0$: Then \eqref{e:j7yG9i-9jmL409h} yields
$\|s^*_n\|^2+\|t_n\|^2=0=\|\boldsymbol{x}_{n+1/2}-
\boldsymbol{x}_n\|^2/(\alpha\varepsilon)^2$.
\item
$\tau_n>0$: Then
it follows from \eqref{e:9g45g2h29a} and 
\eqref{e:j7yG9i-9jmL409h} that 
\begin{align}
\label{e:j7yG9i-9jmL409s}
\|s^*_n\|^2+\|t_n\|^2
&=\tau_n\nonumber\\
&\leq\frac{\big(\scal{x_n-a_n}{a_n^*+L^*v_n^*}+
\scal{Lx_n-b_n}{b_n^*-v_n^*}\big)^2}{\alpha^2\tau_n}\nonumber\\
&=\frac{\big(\scal{x_n}{s^*_n}+\scal{t_n}{v^*_n}-\scal{a_n}{a^*_n}
-\scal{b_n}{b^*_n}\big)^2}{\alpha^2\tau_n}\nonumber\\
&\leq\frac{\lambda_n^2
\big(\scal{x_n}{s^*_n}+\scal{t_n}{v^*_n}-\scal{a_n}{a^*_n}-
\scal{b_n}{b^*_n}\big)^2}{\alpha^2\varepsilon^2\tau_n}\nonumber\\
&=\frac{\theta_n^2\tau_n}{\alpha^2\varepsilon^2}\nonumber\\
&=\frac{\|x_{n+1/2}-x_n\|^2+\|v^*_{n+1/2}-v^*_n\|^2}
{\alpha^2\varepsilon^2}\nonumber\\
&=\frac{\|\boldsymbol{x}_{n+1/2}-\boldsymbol{x}_n\|^2}
{\alpha^2\varepsilon^2}.
\end{align}
\end{itemize}
Altogether, it follows from 
Proposition~\ref{p:j7yG9i-9jmL406}\ref{p:j7yG9i-9jmL406iii}
that $\sum_{n\in\NN}\|s^*_n\|^2+\sum_{n\in\NN}\|t_n\|^2<\pinf$.

\ref{t:9i-9jmL402iii}:
Take $x\in\HH$, $v^*\in\GG$, and a strictly increasing sequence 
$(k_n)_{n\in\NN}$ in $\NN$, such that $x_{k_n}\weakly x$ and 
$v^*_{k_n}\weakly v^*$. We derive from \ref{t:9i-9jmL402ii} and
\eqref{e:j7yG9i-9jmL409h} that $a^*_n+L^*b^*_n\to 0$ and 
$La_n-b_n\to 0$. Hence, the assumptions yield
\begin{equation}
\label{e:j7yG9i-9jmL411a}
a_{k_n}\weakly{x},\quad
b^*_{k_n}\weakly{v^*},\quad
a^*_{k_n}+L^*b^*_{k_n}\to 0,\quad\text{and}\quad
La_{k_n}-b_{k_n}\to 0.
\end{equation}
On the other hand, \eqref{e:9g45g2h29a} also asserts that
$(\forall n\in\NN)$ $(a_n,a_n^*)\in\gra A$ and 
$(b_n,b_n^*)\in\gra B$. Altogether, 
Proposition~\ref{p:genna3Hbl-915} implies that 
$(x,v^*)\in\boldsymbol{Z}$. In view of 
Proposition~\ref{p:j7yG9i-9jmL406}\ref{p:j7yG9i-9jmL406iv},
the proof is complete.
\end{proof}

\begin{remark}
\label{r:9i-9jmL419}
Here are a few observations pertaining to 
Theorem~\ref{t:9i-9jmL402}.
\begin{enumerate}
\item
These results appear to provide the first algorithmic framework for 
composite inclusions problems that does not require additional 
assumptions on the constituents of the problem to achieve strong 
convergence. 
\item
If the second half of \eqref{e:j7yG9i-9jmL409h} is 
by-passed, i.e., if we set
$x_{n+1}=x_{n+1/2}$ and $v^*_{n+1}=v^*_{n+1/2}$, and if the
relaxation parameter $\lambda_n$ is chosen in the range
$[\varepsilon,2-\varepsilon]$, one recovers the
algorithm of \cite[Corollary~3.3]{Genn13}. However, this 
algorithm provides only weak convergence to an unspecified
Kuhn-Tucker point, whereas \eqref{e:j7yG9i-9jmL409h} 
guarantees strong convergence to the best Kuhn-Tucker 
approximation to $(x_0,v_0^*)$. This can be viewed as 
another manifestation of the weak-to-strong convergence principle 
investigated in \cite{Moor01} in a different setting 
($\Tt$-class operators).
\end{enumerate}
\end{remark}

The following proposition is an application of 
Theorem~\ref{t:9i-9jmL402} which describes a concrete implementation
of \eqref{e:j7yG9i-9jmL409h} with a specific rule for selecting
$(a_n,b_n,a_n^*,b_n^*)\in\boldsymbol{G}_\alpha(x_n,v_n^*)$.

\begin{proposition}
\label{p:j7yG9i-9jmL407}
Consider the setting of Problem~\ref{prob:11}. 
Let $\varepsilon\in\zeroun$ and iterate
\begin{equation}
\label{e:9i-9jmL403a}
\begin{array}{l}
\text{for}\;n=0,1,\ldots\\
\left\lfloor
\begin{array}{l}
(\gamma_n,\mu_n)\in [\varepsilon,1/\varepsilon]^2\\
a_n=J_{\gamma_n A}(x_n-\gamma_n L^*v_n^*)\\
l_n=Lx_n\\
b_n=J_{\mu_n B}(l_n+\mu_n v_n^*)\\
s^*_n=\gamma_n^{-1}(x_n-a_n)+\mu_n^{-1}L^*(l_n-b_n)\\
t_n=b_n-La_n\\
\tau_n=\|s_n^*\|^2+\|t_n\|^2\\
\text{if}\;\tau_n=0\\
\left\lfloor
\begin{array}{l}
\theta_n=0\\
\end{array}
\right.\\
\text{if}\;\tau_n>0\\
\left\lfloor
\begin{array}{l}
\lambda_n\in\left[\varepsilon,1\right]\\
\theta_n=\lambda_n\big(\gamma_n^{-1}\|x_n-a_n\|^2+\mu_n^{-1}
\|l_n-b_n\|^2\big)/\tau_n\\
\end{array}
\right.\\
x_{n+1/2}=x_n-\theta_n s^*_n\\
v^*_{n+1/2}=v^*_n-\theta_n t_n\\
\chi_n=\scal{x_0-x_n}{x_n-x_{n+1/2}}
+\scal{v_0^*-v_n^*}{v_n^*-v_{n+1/2}^*}\\
\mu_n=\|x_0-x_n\|^2+\|v_0^*-v_n^*\|^2\\
\nu_n=\|x_n-x_{n+1/2}\|^2+\|v_n^*-v_{n+1/2}^*\|^2\\
\rho_n=\mu_n\nu_n-\chi_n^2\\
\text{if}\;\rho_n=0\;\text{and}\;\chi_n\geq 0\\
\left\lfloor
\begin{array}{l}
x_{n+1}=x_{n+1/2}\\
v^*_{n+1}=v_{n+1/2}^*
\end{array}
\right.\\
\text{if}\;\rho_n>0\;\text{and}\;\chi_n\nu_n\geq\rho_n\\
\left\lfloor
\begin{array}{l}
x_{n+1}=x_0+(1+\chi_n/\nu_n)(x_{n+1/2}-x_n)\\
v^*_{n+1}=v_0^*+(1+\chi_n/\nu_n)(v_{n+1/2}^*-v_n^*)
\end{array}
\right.\\
\text{if}\;\rho_n>0\;\text{and}\;\chi_n\nu_n<\rho_n\\
\left\lfloor
\begin{array}{l}
x_{n+1}=x_n+(\nu_n/\rho_n)\big(\chi_n(x_0-x_n)
+\mu_n(x_{n+1/2}-x_n)\big)\\
v^*_{n+1}=v_n^*+(\nu_n/\rho_n)\big(\chi_n(v_0^*-v_n^*)
+\mu_n(v_{n+1/2}^*-v_n^*)\big).
\end{array}
\right.\\
\end{array}
\right.\\
\end{array}
\end{equation}
Then \eqref{e:9i-9jmL403a} generates infinite sequences 
$(x_n)_{n\in\NN}$ and $(v_n^*)_{n\in\NN}$, and the following
hold:
\begin{enumerate}
\item
\label{p:j7yG9i-9jmL407i}
$\sum_{n\in\NN}\|x_{n+1}-x_n\|^2<\pinf$ and 
$\sum_{n\in\NN}\|v^*_{n+1}-v^*_n\|^2<\pinf$.
\item
\label{p:j7yG9i-9jmL407ii}
$\sum_{n\in\NN}\|s^*_n\|^2<\pinf$ and 
$\sum_{n\in\NN}\|t_n\|^2<\pinf$.
\item
\label{p:j7yG9i-9jmL407iii}
$\sum_{n\in\NN}\|x_n-a_n\|^2<\pinf$ and 
$\sum_{n\in\NN}\|Lx_n-b_n\|^2<\pinf$.
\item
\label{p:j7yG9i-9jmL407iv}
$x_n\to\overline{x}$ and $v_n^*\to\overline{v}^*$.
\end{enumerate}
\end{proposition}
\begin{proof}
Let us define 
\begin{equation}
\label{e:j7yG69gjn.-19o}
\alpha=\frac{\varepsilon}{1+\|L\|^2+2(1-\varepsilon^{2})
\text{max}\big\{1,\|L\|^2\big\}}
\end{equation}
and
\begin{equation}
\label{e:j7yG69gjn.-19r}
(\forall n\in\NN)\quad
a^*_n=\gamma_n^{-1}(x_n-a_n)-L^*v_n^*\quad\text{and}\quad
b^*_n=\mu_n^{-1}(Lx_n-b_n)+v_n^*.
\end{equation}
Then it is shown in \cite[proof of Proposition~3.5]{Genn13} that 
\begin{equation}
\label{e:9i-9jmL419e}
(\forall n\in\NN)\quad
(a_n,b_n,a_n^*,b_n^*)\in\boldsymbol{G}_\alpha(x_n,v_n^*)
\end{equation}
and
\begin{equation}
\label{e:9i-9jmL419f}
(\forall n\in\NN)\quad\|x_n-a_n\|^2\leq 2\varepsilon^{-2}
\big(\|s^*_n\|^2+\varepsilon^{-2}\|L\|^2\,\|t_n\|^2\big).
\end{equation}
We deduce from \eqref{e:j7yG69gjn.-19r} and 
\eqref{e:9i-9jmL419e} that \eqref{e:9i-9jmL403a} is a special case 
of \eqref{e:j7yG9i-9jmL409h}. Consequently, assertions
\ref{p:j7yG9i-9jmL407i} and \ref{p:j7yG9i-9jmL407ii} follow 
from their counterparts in Theorem~\ref{t:9i-9jmL402}. To show 
\ref{p:j7yG9i-9jmL407iii} it suffices to note that 
\eqref{e:9i-9jmL419f} and  \ref{p:j7yG9i-9jmL407ii} imply that
\begin{equation}
\label{e:293}
\sum_{n\in\NN}\|x_n-a_n\|^2<\pinf 
\end{equation}
and hence that
$\sum_{n\in\NN}\|Lx_n-b_n\|^2<\pinf$ since\begin{equation}
\label{e:69gjn.-08g}
(\forall n\in\NN)\quad
\|Lx_n-b_n\|^2=\|L(x_n-a_n)+La_n-b_n\|^2\leq 
2\big(\|L\|^2\,\|x_n-a_n\|^2+\|t_n\|^2\big).
\end{equation}
In turn, \eqref{e:j7yG69gjn.-19r} yields
\begin{equation}
\label{e:294}
\sum_{n\in\NN}\|v_n^*-b_n^*\|^2
=\sum_{n\in\NN}\mu_n^{-2}\|Lx_n-b_n\|^2
\leq\varepsilon^{-2}\sum_{n\in\NN}\|Lx_n-b_n\|^2<\pinf.
\end{equation}
Altogether, \ref{p:j7yG9i-9jmL407iv} follows from 
\eqref{e:293}, \eqref{e:294}, and 
Theorem~\ref{t:9i-9jmL402}\ref{t:9i-9jmL402iii}.
\end{proof}

\begin{remark}
\label{r:9i-9jmL429}
In \eqref{e:9i-9jmL403a}, the identity $\tau_n=0$ can be used as 
a stopping rule. Indeed, $\tau_n=0$ $\Leftrightarrow$
$(a_n^*+L^*b_n^*,b_n-La_n)=(0,0)$ $\Leftrightarrow$
$(-L^*b_n^*,La_n)=(a_n^*,b_n)\in Aa_n\times B^{-1}b_n^*$
$\Leftrightarrow$ $(a_n,b_n^*)\in\boldsymbol{Z}$. On the other hand,
it follows from \eqref{e:9i-9jmL419f} and 
\eqref{e:69gjn.-08g} that $\tau_n=0$ $\Rightarrow$
$(x_n,v_n^*)=(a_n,b_n^*)$. Altogether, Remark~\ref{r:9i-9jmL428} 
yields $(x_n,v_n^*)=P_{\boldsymbol{Z}}(x_0,v_0^*)=
(\overline{x},\overline{v}^*)$.
\end{remark}

\begin{remark}
An important feature of algorithm \eqref{e:9i-9jmL403a} which is 
inherited from that of \cite[Proposition~3.5]{Genn13} is 
that it does not require the knowledge of $\|L\|$ or necessitate 
potentially hard to implement inversions of linear operators. 
\end{remark}

\section{Application to systems of monotone inclusions}
\label{sec:4}

As discussed in 
\cite{Genn13,Sico10,Atto11,Bot13b,Bric13,Siop13,Juan12}, 
various problems in applied mathematics can be modeled by 
systems of coupled monotone inclusions. In this section, we
consider the following setting.

\begin{problem}
\label{prob:12}
Let $m$ and $K$ be strictly positive integers, let
$(\HH_i)_{1\leq i\leq m}$ and $(\GG_k)_{1\leq k\leq K}$ be real 
Hilbert spaces, and set 
$\KKK=\HH_1\oplus\cdots\HH_m\oplus\GG_1\oplus\cdots\oplus\GG_K$.
For every $i\in\{1,\ldots,m\}$ and every $k\in\{1,\ldots,K\}$, let 
$A_i\colon\HH_i\to 2^{\HH_i}$ and $B_k\colon\GG_k\to 2^{\GG_k}$ 
be maximally monotone, let $z_i\in\HH_i$, let $r_k\in\GG_k$, 
and let $L_{ki}\colon\HH_i\to\GG_k$ be linear and bounded. Let 
$(\boldsymbol{x}_0,\boldsymbol{v}_0^*)=
(x_{1,0},\ldots,x_{m,0},v_{1,0}^*,\ldots,v_{K,0}^*)\in\KKK$, 
assume that the coupled inclusions problem
\begin{multline}
\label{e:lk87b'kk-24p}
\text{find}\;\;\overline{x}_1\in\HH_1,\ldots,\overline{x}_m\in\HH_m
\;\;\text{such that}\\
(\forall i\in\{1,\ldots,m\})\quad
z_i\in A_i\overline{x}_i+\Sum_{k=1}^KL_{ki}^*
\bigg(B_k\bigg(\Sum_{j=1}^mL_{kj}\overline{x}_j-r_k\bigg)\bigg)
\end{multline}
has at least one solution, and consider the dual problem 
\begin{multline}
\label{e:lk87b'kk-24d}
\text{find}\;\;\overline{v}_1^*\in\GG_1,\ldots,\overline{v}^*_K
\in\GG_K
\;\;\text{such that}\\
(\forall k\in\{1,\ldots,K\})\quad
-r_k\in-\Sum_{i=1}^mL_{ki}\bigg(A_i^{-1}
\bigg(z_i-\Sum_{l=1}^KL_{li}^*\overline{v}^*_l\bigg)\bigg)
+B_k^{-1}\overline{v}^*_k.
\end{multline}
The problem is to find the best approximation 
$(\overline{x}_1,\ldots,\overline{x}_m,\overline{v}_1^*,\ldots,
\overline{v}_K^*)$ to 
$(\boldsymbol{x}_0,\boldsymbol{v}_0^*)$ from the 
associated Kuhn-Tucker set
\begin{multline}
\label{e:9g45g2h07k}
\boldsymbol{Z}=\bigg\{(x_1,\ldots,x_m,v_1^*,\ldots,v^*_K)\in\KKK
\;\bigg |\;
(\forall i\in\{1,\ldots,m\})\;\;z_i-\sum_{k=1}^KL_{ki}^*v_k^*\in
A_ix_i\:\;\text{and}\\
(\forall k\in\{1,\ldots,K\})\;\;\sum_{i=1}^mL_{ki}x_i-r_k\in
B_k^{-1}v_k^*\bigg\}.
\end{multline}
\end{problem}

The next result presents a strongly convergent method for solving
Problem~\ref{prob:12}. Let us note that existing methods require 
stringent additional conditions on the operators to achieve strong 
convergence, produce only unspecified points in the Kuhn-Tucker 
set, and necessitate the knowledge of the norms of the linear 
operators present in the model \cite{Sico10,Siop13}. These
shortcomings are simultaneously circumvented in the proposed 
algorithm.

\begin{proposition}
\label{p:9g45g2h07}
Consider the setting of Problem~\ref{prob:12}. 
Let $\varepsilon\in\zeroun$ and iterate
\begin{equation}
\label{e:j7yG9i-9jmL407a}
\begin{array}{l}
\text{for}\;n=0,1,\ldots\\
\left\lfloor
\begin{array}{l}
(\gamma_n,\mu_n)\in [\varepsilon,1/\varepsilon]^2\\
\text{for}\;i=1,\ldots,m\\
\left\lfloor
\begin{array}{l}
a_{i,n}=J_{\gamma_n A_i}\big(x_{i,n}+\gamma_n
\big(z_i-\sum_{k=1}^KL_{ki}^*v_{k,n}^*\big)\big)\\
\end{array}
\right.\\
\text{for}\;k=1,\ldots,K\\
\left\lfloor
\begin{array}{l}
l_{k,n}=\sum_{i=1}^mL_{ki}x_{i,n}\\
b_{k,n}=r_k+J_{\mu_n B_k}\big(l_{k,n}+\mu_nv_{k,n}^*-r_k\big)\\
t_{k,n}=b_{k,n}-\sum_{i=1}^mL_{ki}a_{i,n}\\
\end{array}
\right.\\
\text{for}\;i=1,\ldots,m\\
\left\lfloor
\begin{array}{l}
s^*_{i,n}=\gamma_n^{-1}(x_{i,n}-a_{i,n})+
\mu_n^{-1}\sum_{k=1}^KL_{ki}^*(l_{k,n}-b_{k,n})\\
\end{array}
\right.\\
\tau_n=\sum_{i=1}^m\|s_{i,n}^*\|^2+\sum_{k=1}^K\|t_{k,n}\|^2\\
\text{if}\;\tau_n=0\\
\left\lfloor
\begin{array}{l}
\theta_n=0\\
\end{array}
\right.\\
\text{if}\;\tau_n>0\\
\left\lfloor
\begin{array}{l}
\lambda_n\in\left[\varepsilon,1\right]\\
\theta_n=\lambda_n\big(\gamma_n^{-1}\sum_{i=1}^m
\|x_{i,n}-a_{i,n}\|^2+\mu_n^{-1}
\sum_{k=1}^K\|l_{k,n}-b_{k,n}\|^2\big)/\tau_n\\
\end{array}
\right.\\
\text{for}\;i=1,\ldots,m\\
\left\lfloor
\begin{array}{l}
x_{i,n+1/2}=x_{i,n}-\theta_n s^*_{i,n}\\
\end{array}
\right.\\
\text{for}\;k=1,\ldots,K\\
\left\lfloor
\begin{array}{l}
v^*_{k,n+1/2}=v^*_{k,n}-\theta_n t_{k,n}
\end{array}
\right.\\
\chi_n=\sum_{i=1}^m\scal{x_{i,0}-x_{i,n}}{x_{i,n}-x_{i,n+1/2}}
+\sum_{k=1}^K\scal{v_{k,0}^*-v_{k,n}^*}{v_{k,n}^*-v_{k,n+1/2}^*}\\
\mu_n=\sum_{i=1}^m\|x_{i,0}-x_{i,n}\|^2+\sum_{k=1}^K
\|v_{k,0}^*-v_{k,n}^*\|^2\\
\nu_n=\sum_{i=1}^m\|x_{i,n}-x_{i,n+1/2}\|^2+
\sum_{k=1}^K\|v_{k,n}^*-v_{k,n+1/2}^*\|^2\\
\rho_n=\mu_n\nu_n-\chi_n^2\\
\text{if}\;\rho_n=0\;\text{and}\;\chi_n\geq 0\\
\left\lfloor
\begin{array}{l}
\text{for}\;i=1,\ldots,m\\
\left\lfloor
\begin{array}{l}
x_{i,n+1}=x_{i,n+1/2}\\
\end{array}
\right.\\
\text{for}\;k=1,\ldots,K\\
\left\lfloor
\begin{array}{l}
v^*_{k,n+1}=v_{k,n+1/2}^*\\
\end{array}
\right.\\
\end{array}
\right.\\
\text{if}\;\rho_n>0\;\text{and}\;\chi_n\nu_n\geq\rho_n\\
\left\lfloor
\begin{array}{l}
\text{for}\;i=1,\ldots,m\\
\left\lfloor
\begin{array}{l}
x_{i,n+1}=x_{i,0}+(1+\chi_n/\nu_n)(x_{i,n+1/2}-x_{i,n})\\
\end{array}
\right.\\
\text{for}\;k=1,\ldots,K\\
\left\lfloor
\begin{array}{l}
v^*_{k,n+1}=v_{k,0}^*+(1+\chi_n/\nu_n)(v_{k,n+1/2}^*-v_{k,n}^*)
\end{array}
\right.\\
\end{array}
\right.\\
\text{if}\;\rho_n>0\;\text{and}\;\chi_n\nu_n<\rho_n\\
\left\lfloor
\begin{array}{l}
\text{for}\;i=1,\ldots,m\\
\left\lfloor
\begin{array}{l}
x_{i,n+1}=x_{i,n}+(\nu_n/\rho_n)\big(\chi_n(x_{i,0}-x_{i,n})
+\mu_n(x_{i,n+1/2}-x_{i,n})\big)\\
\end{array}
\right.\\
\text{for}\;k=1,\ldots,K\\
\left\lfloor
\begin{array}{l}
v^*_{k,n+1}=v_{k,n}^*+(\nu_n/\rho_n)
\big(\chi_n(v_{k,0}^*-v_{k,n}^*)
+\mu_n(v_{k,n+1/2}^*-v_{k,n}^*)\big).
\end{array}
\right.\\
\end{array}
\right.\\
\end{array}
\right.\\
\end{array}
\end{equation}
Then \eqref{e:j7yG9i-9jmL407a} generates infinite sequences
$(x_{1,n})_{n\in\NN}$, \ldots, $(x_{m,n})_{n\in\NN}$,
$(v_{1,n}^*)_{n\in\NN}$, \ldots, $(v_{K,n}^*)_{n\in\NN}$,
and the following hold:
\begin{enumerate}
\item
\label{p:9g45g2h07i}
Let $i\in\{1,\ldots,m\}$. Then
$\sum_{n\in\NN}\!\|s^*_{i,n}\|^2\!<\!\pinf$,
$\sum_{n\in\NN}\!\|x_{i,n+1}-x_{i,n}\|^2\!<\!\pinf$,
$\sum_{n\in\NN}\!\|x_{i,n}-a_{i,n}\|^2\!<\!\pinf$, and
$x_{i,n}\to\overline{x}_i$.
\item
\label{p:9g45g2h07ii}
Let $k\in\{1,\ldots,K\}$. Then
$\sum_{n\in\NN}\!\|t_{k,n}\|^2\!<\!\pinf$,
$\sum_{n\in\NN}\!\|v^*_{k,n+1}-v^*_{k,n}\|^2\!<\!\pinf$,
$\sum_{n\in\NN}\!\|\sum_{i=1}^mL_{ki}x_{i,n}-b_{k,n}\|^2\!<\!\pinf$,
and $v^*_{k,n}\to\overline{v}_k^*$. 
\end{enumerate}
\end{proposition}
\begin{proof}
Let us set $\HH=\bigoplus_{i=1}^m\HH_i$ and 
$\GG=\bigoplus_{k=1}^K\GG_k$, and let us introduce the operators
\begin{equation}
\label{e:9g45g2h10a}
\begin{cases}
A\colon\HH\to 2^{\HH}\colon (x_i)_{1\leq i\leq m}\mapsto
\cart_{\!i=1}^{\!m}(-z_i+A_ix_i)\\
B\colon\GG\to 2^{\GG}\colon
(y_k)_{1\leq k\leq K}\mapsto\cart_{\!k=1}^{\!K}B_k(y_k-r_k)\\
L\colon\HH\to\GG\colon (x_i)_{1\leq i\leq m}\mapsto 
\big(\sum_{i=1}^mL_{ki}x_i\big)_{1\leq k\leq K}.
\end{cases}
\end{equation}
Then $L^*\colon\GG\to\HH\colon (y_k)_{1\leq k\leq K}
\mapsto(\sum_{k=1}^KL_{ki}^*y_k)_{1\leq i\leq m}$ and, in this 
setting, Problem~\ref{prob:11} becomes
Problem~\ref{prob:12}. Next, for every $n\in\NN$, let us introduce 
the variables 
$a_n=(a_{i,n})_{1\leq i\leq m}$, 
$s^*_n=(s^*_{i,n})_{1\leq i\leq m}$, 
$x_n=(x_{i,n})_{1\leq i\leq m}$, 
$x_{n+1/2}=(x_{i,n+1/2})_{1\leq i\leq m}$, 
$b_n=(b_{k,n})_{1\leq k\leq K}$, 
$l_n=(l_{k,n})_{1\leq k\leq K}$, 
$t_n=(t_{k,n})_{1\leq k\leq K}$,
$v_n^*=(v^*_{k,n})_{1\leq k\leq K}$, and
$v_{n+1/2}^*=(v^*_{k,n+1/2})_{1\leq k\leq K}$. 
Since \cite[Propositions~23.15 and 23.16]{Livre1} assert that
\begin{multline}
(\forall n\in\NN)(\forall (x_i)_{1\leq i\leq m}\in\HH)
(\forall (y_k)_{1\leq k\leq K}\in\GG)\quad
J_{\gamma_n A}(x_i)_{1\leq i\leq m}=\big(J_{\gamma_n A_i}
(x_{i}+\gamma_n z_i)\big)_{1\leq i\leq m}\\
\text{and}\quad J_{\mu_n B}(y_k)_{1\leq k\leq K}=
\big(r_k+J_{\mu_n B_k}(y_{k}-r_k)\big)_{1\leq k\leq K},
\end{multline}
\eqref{e:9i-9jmL403a} reduces in the present scenario to
\eqref{e:j7yG9i-9jmL407a}. Thus, the results follow from 
Proposition~\ref{p:j7yG9i-9jmL407}.
\end{proof}

\begin{example}
\label{ex:9i-9jmL421}
Let $A$, $(B_k)_{1\leq k\leq K}$, and $(S_k)_{1\leq k\leq K}$ be 
maximally monotone operators acting on a real Hilbert space $\HH$. 
We revisit a problem discussed in \cite[Section~4]{Siop13}, namely
the relaxation of the possibly inconsistent inclusion problem 
\begin{equation}
\label{e:cras1995}
\text{find}\;\;\overline{x}\in\HH\;\;\text{such that}\;\;
0\in A\overline{x}\cap\bigcap_{k=1}^KB_k\overline{x}
\end{equation}
to
\begin{equation}
\label{e:2012-11-29p}
\text{find}\;\;\overline{x}\in\HH\;\;\text{such that}\;\;
0\in A\overline{x}+\sum_{k=1}^K(B_k\infconv S_k)\overline{x},
\quad\text{where}\quad B_k\infconv S_k=
\big(B^{-1}_k+S^{-1}_k\big)^{-1}.
\end{equation}
We assume that \eqref{e:2012-11-29p} has at least one solution 
and that, for every $k\in\{1,\ldots,K\}$, $S_k^{-1}$ is at most 
single-valued and strictly monotone, with $S_k^{-1} 0=\{0\}$. 
Hence, \eqref{e:2012-11-29p} is a relaxation of 
\eqref{e:cras1995} in the sense that if the latter happens to
have solutions, they coincide with those of the former 
\cite[Proposition~4.2]{Siop13}. As shown in \cite{Siop13}, this 
framework captures many relaxation schemes, and a point 
$\overline{x}_1\in\HH$ solves \eqref{e:2012-11-29p} if 
and only if $(\overline{x}_1,\overline{x}_2,\ldots,\overline{x}_m)$ 
solves \eqref{e:lk87b'kk-24p}, where $m=K+1$, $\HH_1=\HH$, $A_1=A$, 
$z_1=0$, and, for every $k\in\{1,\ldots,K\}$,  
\begin{equation}
\label{e:9i-9jmL423}
\begin{cases}
\HH_{k+1}=\HH\\
\GG_{k}=\HH\\
A_{k+1}=S_k\\
z_{k+1}=0\\
r_k=0
\end{cases}
\qquad\text{and}\quad
\begin{cases}
L_{k1}=\Id\\
(\forall i\in\{2,\ldots,m\})\:\; L_{ki}=
\begin{cases}
-\Id,&\text{if}\;\;i=k+1;\\
0,&\text{otherwise.}
\end{cases}
\end{cases}
\end{equation}
Thus \eqref{e:j7yG9i-9jmL407a} can be reduced to
\begin{equation}
\label{e:9i-9jmL422a}
\begin{array}{l}
\text{for}\;n=0,1,\ldots\\
\left\lfloor
\begin{array}{l}
(\gamma_n,\mu_n)\in [\varepsilon,1/\varepsilon]^2\\
a_{1,n}=J_{\gamma_n A}\big(x_{1,n}-\gamma_n
\sum_{k=1}^Kv_{k,n}^*\big)\\
\text{for}\;k=1,\ldots,K\\
\left\lfloor
\begin{array}{l}
a_{k+1,n}=J_{\gamma_n S_k}\big(x_{k+1,n}+\gamma_n v_{k,n}^*\big)\\
l_{k,n}=x_{1,n}-x_{k+1,n}\\
b_{k,n}=J_{\mu_n B_k}\big(l_{k,n}+\mu_nv_{k,n}^*\big)\\
t_{k,n}=b_{k,n}+a_{k+1,n}-a_{1,n}\\
s^*_{k+1,n}=\gamma_n^{-1}(x_{k+1,n}-a_{k+1,n})+
\mu_n^{-1}(b_{k,n}-l_{k,n})\\
\end{array}
\right.\\
s^*_{1,n}=\gamma_n^{-1}(x_{1,n}-a_{1,n})+
\mu_n^{-1}\sum_{k=1}^K(l_{k,n}-b_{k,n})\\
\tau_n=\sum_{k=1}^{K+1}\|s_{k,n}^*\|^2+\sum_{k=1}^K\|t_{k,n}\|^2\\
\text{if}\;\tau_n=0\\
\left\lfloor
\begin{array}{l}
\theta_n=0\\
\end{array}
\right.\\
\text{if}\;\tau_n>0\\
\left\lfloor
\begin{array}{l}
\lambda_n\in\left[\varepsilon,1\right]\\
\theta_n=\lambda_n\big(\gamma_n^{-1}\sum_{k=1}^{K+1}
\|x_{k,n}-a_{k,n}\|^2+\mu_n^{-1}
\sum_{k=1}^K\|l_{k,n}-b_{k,n}\|^2\big)/\tau_n\\
\end{array}
\right.\\
x_{1,n+1/2}=x_{1,n}-\theta_n s^*_{1,n}\\
\text{for}\;k=1,\ldots,K\\
\left\lfloor
\begin{array}{l}
x_{k+1,n+1/2}=x_{k+1,n}-\theta_n s^*_{k+1,n}\\
v^*_{k,n+1/2}=v^*_{k,n}-\theta_n t_{k,n}
\end{array}
\right.\\
\chi_n=\sum_{k=1}^{K+1}\scal{x_{k,0}-x_{k,n}}{x_{k,n}-x_{k,n+1/2}}
+\sum_{k=1}^K\scal{v_{k,0}^*-v_{k,n}^*}{v_{k,n}^*-v_{k,n+1/2}^*}\\
\mu_n=\sum_{k=1}^{K+1}\|x_{k,0}-x_{k,n}\|^2+\sum_{k=1}^K
\|v_{k,0}^*-v_{k,n}^*\|^2\\
\nu_n=\sum_{k=1}^{K+1}\|x_{k,n}-x_{k,n+1/2}\|^2+
\sum_{k=1}^K\|v_{k,n}^*-v_{k,n+1/2}^*\|^2\\
\rho_n=\mu_n\nu_n-\chi_n^2\\
\text{if}\;\rho_n=0\;\text{and}\;\chi_n\geq 0\\
\left\lfloor
\begin{array}{l}
x_{1,n+1}=x_{1,n+1/2}\\
\text{for}\;k=1,\ldots,K\\
\left\lfloor
\begin{array}{l}
x_{k+1,n+1}=x_{k+1,n+1/2}\\
v^*_{k,n+1}=v_{k,n+1/2}^*\\
\end{array}
\right.\\
\end{array}
\right.\\
\text{if}\;\rho_n>0\;\text{and}\;\chi_n\nu_n\geq\rho_n\\
\left\lfloor
\begin{array}{l}
x_{1,n+1}=x_{1,0}+(1+\chi_n/\nu_n)(x_{1,n+1/2}-x_{1,n})\\
\text{for}\;k=1,\ldots,K\\
\left\lfloor
\begin{array}{l}
x_{k+1,n+1}=x_{k+1,0}+(1+\chi_n/\nu_n)(x_{k+1,n+1/2}-x_{k+1,n})\\
v^*_{k,n+1}=v_{k,0}^*+(1+\chi_n/\nu_n)(v_{k,n+1/2}^*-v_{k,n}^*)
\end{array}
\right.\\
\end{array}
\right.\\
\text{if}\;\rho_n>0\;\text{and}\;\chi_n\nu_n<\rho_n\\
\left\lfloor
\begin{array}{l}
x_{1,n+1}=x_{1,n}+(\nu_n/\rho_n)\big(\chi_n(x_{1,0}-x_{1,n})
+\mu_n(x_{1,n+1/2}-x_{1,n})\big)\\
\text{for}\;k=1,\ldots,K\\
\left\lfloor
\begin{array}{l}
x_{k+1,n+1}=x_{k+1,n}+(\nu_n/\rho_n)\big(\chi_n(x_{k+1,0}-x_{k+1,n})
+\mu_n(x_{k+1,n+1/2}-x_{k+1,n})\big)\\
v^*_{k,n+1}=v_{k,n}^*+(\nu_n/\rho_n)
\big(\chi_n(v_{k,0}^*-v_{k,n}^*)
+\mu_n(v_{k,n+1/2}^*-v_{k,n}^*)\big),
\end{array}
\right.\\
\end{array}
\right.\\
\end{array}
\right.\\
\end{array}
\end{equation}
and it follows from Proposition~\ref{p:9g45g2h07} that 
$(x_{1,n})_{n\in\NN}$ converges strongly to a solution 
$\overline{x}_1$ to the relaxed problem \eqref{e:2012-11-29p}. Let 
us note that the algorithm proposed in 
\cite[Proposition~4.2]{Siop13} to solve \eqref{e:2012-11-29p}
requires that $A$ be uniformly monotone at $\overline{x}_1$ to 
guarantee strong convergence, whereas this assumption is not needed 
here. In addition, the scaling parameters used in the resolvents of 
the monotone operators in \cite[Proposition~4.2]{Siop13} must be 
identical at each iteration and bounded
by a fixed constant: $(\forall n\in\NN)$ 
$\gamma_n=\mu_n\in[\varepsilon,(1-\varepsilon)/\sqrt{K+1}]$. By 
contrast, the parameters $\mu_n$ and $\gamma_n$ in
\eqref{e:9i-9jmL422a} may differ and they can be 
arbitrarily large since $\varepsilon$ can be arbitrarily small, 
which could have some beneficial impact in terms of speed of 
convergence.
\end{example}

As a second illustration of Proposition~\ref{p:9g45g2h07}, we 
consider the following multivariate minimization problem.

\begin{problem}
\label{prob:13}
Let $m$ and $K$ be strictly positive integers, let
$(\HH_i)_{1\leq i\leq m}$ and $(\GG_k)_{1\leq k\leq K}$ be real 
Hilbert spaces, and set 
$\KKK=\HH_1\oplus\cdots\HH_m\oplus\GG_1\oplus\cdots\oplus\GG_K$.
For every $i\in\{1,\ldots,m\}$ and every $k\in\{1,\ldots,K\}$, let 
$f_i\in\Gamma_0(\HH_i)$ and $g_k\in\Gamma_0(\GG_k)$,
let $z_i\in\HH_i$, let $r_k\in\GG_k$, 
and let $L_{ki}\colon\HH_i\to\GG_k$ be linear and bounded. Let 
$(\boldsymbol{x}_0,\boldsymbol{v}_0^*)=
(x_{1,0},\ldots,x_{m,0},v_{1,0}^*,\ldots,v_{K,0}^*)\in\KKK$ and
assume that 
\begin{equation}
\label{e:2012-10-21a}
(\forall i\in\{1,\ldots,m\})\quad
z_i\in\ran\bigg(\partial f_i+\sum_{k=1}^KL_{ki}^*\circ
\partial g_k\circ\bigg(\sum_{j=1}^mL_{kj}\cdot-r_k\bigg)\bigg).
\end{equation}
Consider the primal problem 
\begin{equation}
\label{e:9i-9jmL420p}
\minimize{x_1\in\HH_1,\ldots,\,x_m\in\HH_m}{\sum_{i=1}^m
\big(f_i(x_i)-\scal{x_i}{z_i}\big)+\sum_{k=1}^K 
g_k\bigg(\sum_{i=1}^mL_{ki}x_i-r_k\bigg)}
\end{equation}
and the dual problem
\begin{equation}
\label{e:9i-9jmL420d}
\minimize{v^*_1\in\GG_1,\ldots,\,v^*_K\in\GG_K}{\sum_{i=1}^m
f_i^*\bigg(z_i-\sum_{k=1}^KL_{ki}^*v^*_k\bigg)
+\sum_{k=1}^K\big(g^*_k(v^*_k)+\scal{v^*_k}{r_k}\big)}.
\end{equation}
The objective is to find the best approximation 
$(\overline{x}_1,\ldots,\overline{x}_m,\overline{v}_1^*,\ldots,
\overline{v}_K^*)$ to 
$(\boldsymbol{x}_0,\boldsymbol{v}_0^*)$ from the 
associated Kuhn-Tucker set
\begin{multline}
\label{e:9i-9jmL420k}
\boldsymbol{Z}=\bigg\{(x_1,\ldots,x_m,v_1^*,\ldots,v^*_K)\in\KKK
\;\bigg |\;
(\forall i\in\{1,\ldots,m\})\;\;z_i-\sum_{k=1}^KL_{ki}^*v_k^*\in
\partial f_i(x_i)\:\;\text{and}\\
(\forall k\in\{1,\ldots,K\})\;\;\sum_{i=1}^mL_{ki}x_i-r_k\in
\partial g_k^*(v_k^*)\bigg\}.
\end{multline}
\end{problem}

The following corollary provides a strongly convergent method to
solve Problem~\ref{prob:13}. Recall that the Moreau proximity 
operator \cite{Mor62b} of a function $\varphi\in\Gamma_0(\HH)$ is
$\prox_\varphi=J_{\partial\varphi}$, i.e., the operator which maps
every point $x\in\HH$ to the unique minimizer of the function 
$y\mapsto\varphi(y)+\|x-y\|^2/2$.

\begin{corollary}
\label{c:9g45g2h13}
Consider the setting of Problem~\ref{prob:13}. 
Let $\varepsilon\in\zeroun$ and execute 
\eqref{e:j7yG9i-9jmL407a}, where $J_{\gamma_n A_i}$ is 
replaced by $\prox_{\gamma_n f_i}$ and 
$J_{\mu_n B_k}$ is replaced by $\prox_{\mu_n g_k}$. Then the 
following hold:
\begin{enumerate}
\item
$(\overline{x}_1,\ldots,\overline{x}_m)$ solves 
\eqref{e:9i-9jmL420p} and 
$(\overline{v}^*_1,\ldots,\overline{v}^*_m)$ 
solves \eqref{e:9i-9jmL420d}.
\item
For every $i\in\{1,\ldots,m\}$, $x_{i,n}\to\overline{x}_i$.
\item
For every $k\in\{1,\ldots,K\}$, $v^*_{k,n}\to\overline{v}^*_k$.
\end{enumerate}
\end{corollary}
\begin{proof}
Let us define 
$(\forall i\in\{1,\ldots,m\})$ $A_i=\partial f_i$ and 
$(\forall k\in\{1,\ldots,K\})$ $B_k=\partial g_k$. Then, as shown 
in the proof of \cite[Proposition~5.4]{Siop13},
\eqref{e:2012-10-21a} implies that Problem~\ref{prob:12} assumes 
the form of Problem~\ref{prob:13} and that Kuhn-Tucker points
provide primal and dual solutions. Hence, applying 
Proposition~\ref{p:9g45g2h07} in this setting yields the claims.
\end{proof}

\end{document}